\documentclass[10pt]{amsart}

\usepackage[dvipdfmx]{graphicx,xcolor} %

\usepackage{amssymb}  %\nexists
\usepackage{amscd}      %\begin{CD}

\newtheorem{theorem}{Theorem}[section]                    
\newtheorem{lemma}[theorem]{Lemma}                    
\newtheorem{proposition}[theorem]{Proposition}  

\theoremstyle{definition}

\newtheorem{example}[theorem]{Example}

\theoremstyle{remark}
\newtheorem{remark}[theorem]{Remark}

\numberwithin{equation}{section}

\begin{document}

\title[B\"{o}ttcher coordinates]{B\"{o}ttcher coordinates at superattracting fixed points of holomorphic skew products} 

\author[K. Ueno]{Kohei Ueno}
\address{Daido University, Nagoya 457-8530, Japan}
\email{k-ueno@daido-it.ac.jp}


\date{}

\dedicatory{}

\begin{abstract}
Let $f : (\mathbb{C}^2, 0) \to ( \mathbb{C}^2, 0)$ be a germ of holomorphic skew product
with a superattracting fixed point at the origin.
If it has a suitable weight, then we can construct a B\"{o}ttcher coordinate 
which conjugates $f$ to the associated monomial map.
This B\"{o}ttcher coordinate is defined on an invariant open set
whose interior or boundary contains the origin.
\end{abstract}

\maketitle

\section{Introduction}

Let $p : (\mathbb{C}, 0) \to  \mathbb{C}, 0)$ be 
a holomorphic germ with a superattracting fixed point at the origin.
Taking an affine conjugate,
we may write $p(z) = z^{\delta} + O(z^{\delta + 1})$, where $\delta \geq 2$.
Let $p_0(z) =  z^{\delta}$.
B\"{o}ttcher's theorem \cite{b} asserts that there is
a conformal function $\varphi_p$ defined on a neighborhood of the origin, 
with $\varphi_p \sim id$, that conjugates $p$ to $p_0$.
Here $\varphi_p \sim id$ means that
the ratio of $\varphi_p$ and $id$ converges to $1$ as $z \to 0$.
This function is called the B\"{o}ttcher coordinate for $p$ at the origin,
and obtained as the limit of 
the compositions of $p_0^{-n}$ and $p^n$,
where $p^n$ denotes the $n$-th iterate of $p$.
The branch of $p_0^{-n}$ is taken such that $p_0^{-n} \circ p_0^n = id$.

B\"{o}ttcher's theorem does not extend 
to higher dimensions entirely as stated in \cite{hp}.
For example, let $f(z,w) = (z^2, w^2 + z^4)$.
Then it has a superattracting fixed point at the origin,
but there is no neighborhood of the origin 
on which $f$ is conjugate to $f_0(z,w) = (z^2, w^2)$
because the critical orbits of $f$ and $f_0$ behave differently.
However, we can completely understand the dynamics of $f$
because it is semiconjugate to $g(z,w) = (z^2, w^2 + 1)$
by $\pi(z,w) = (z, z^2 w)$: $\pi \circ g = f \circ \pi$.
In particular,
from the one-dimensional B\"{o}ttcher coordinate for $w \to w^2 + 1$ near infinity,
one can construct a biholomorphic map 
defined on $\{ |z| < r |w|^2 \}$ for small $r$
that conjugates $f$ to $f_0$.
This domain is not a neighborhood of the origin, 
but its boundary contains the origin.
In this paper
we analyze such phenomena for holomorphic skew products
with superattracting fixed points at the origin in $\mathbb{C}^2$.
By assigning suitable weights,
we obtain an analogue of the one-dimensional B\"{o}ttcher coordinates;
see Theorems \ref{main thm: d > 1} and \ref{main thm: d = 1} below.
The idea of this study is the same as that of our previous study \cite{u},
in which we obtained similar results on B\"{o}ttcher coordinates
for polynomial skew products near infinity.
%%%%%%%%%%%
Moreover,
our results are closely related to Theorem 5.1 in \cite{fj},
which is obtained by Theorem C in \cite{fj} and the result in \cite{f}.
Favre and Jonsson \cite{fj} have established 
a systematic way to study the dynamics of 
all holomorphic germs with superattracting fixed points in dimension two;
see also Section 8 in a survey article \cite{j}.
Favre \cite{f} has classified contracting rigid germs in dimension two; 
a germ is called rigid if the union of the critical set of all its iterates is 
a divisor with normal crossing and forward invariant.
See also \cite{r-rigid} and \cite{r}.

For other studies on B\"{o}ttcher's theorem in higher dimensions,
we refer to \cite{ushiki}, \cite{ueda} and \cite{bek};
they dealt with holomorphic germs 
with superattracting fixed points at the origin 
in dimension two or more.
Ushiki \cite{ushiki} and Ueda \cite{ueda} gave
different classes of germs that have 
the B\"{o}ttcher coordinates on neighborhoods of the origin.
Buff, Epstein and Koch \cite{bek} gave 
criteria, in terms of vector fields, 
for a certain class of germs 
to have the B\"{o}ttcher coordinates on neighborhoods of the origin.
The germs in \cite{ushiki} are rigid and conjugate to monomial maps,
whereas the germs in \cite{ueda} or \cite{bek} are conjugate to homogeneous or quasihomogeneous maps.
In addition, we refer to a survey article \cite{a}.
Besides theorems for the superattracting case,
Abate \cite{a} collected major theorems on local dynamics
of holomorphic germs with fixed points of several types 
in one and higher dimensions.

Let $f : ( \mathbb{C}^2, 0) \to (\mathbb{C}^2, 0)$ be 
a holomorphic germ of the form $f(z,w)=(p(z)$, $q(z,w))$,
which is called a holomorphic skew product in this paper.
We assume that it has a superattracting fixed point at the origin;
that is, $f(0) = 0$ and $Df(0)$ is the zero matrix.
Then we may write $p(z) = z^{\delta} + O(z^{\delta + 1})$,
where $\delta \geq 2$.
On the other hand, let
\[ 
q(z,w) = b z^{\gamma} w^d + \sum_{} b_j z^{n_j} w^{m_j},
\]
where $b \neq 0$, $n_j \geq \gamma$, and $m_j > d$ if $n_j = \gamma$.
In other words,
$(\gamma, d)$ is the minimal exponent with respect to the lexicographic order 
that appears in the power series expansion of $q$.
Since the origin is superattracting,
$\gamma + d \geq 2$ and $n_j + m_j \geq 2$.
If $d \geq 2$, then we may assume that $b = 1$.
In this paper
we say that $f$ is \textit{trivial} if $m_j \geq d$ for any $j$.
For this case, we prove that
the B\"{o}ttcher coordinate for $f$ exists on a  neighborhood of the origin,
and the proof is rather easy.
As a remark, $f$ is rigid
if it is \textit{trivial} or $d = 1$.
Moreover,
$f$ belongs to Class 6 in \cite{f} 
and the result follows if it is \textit{trivial} and $d \geq 2$,
and $f$ belongs to Class 4 if $d = 1$.
On the other hand,
we say that $f$ is \textit{non-trivial} if $m_j < d$ for some $j$. 
This case is the difficult part,
in which we need the idea of assigning a suitable weight.

We define the rational number $\alpha$ associated with $f$ as
\[
\alpha =
\min 
\left\{ a \geq 0 \ \Big| 
\begin{array}{lcr}
a \gamma + d \leq \delta \text{ and } 
a \gamma + d \leq a n_j + m_j \\  
\text{for any $j$ such that } b_j \neq 0
\end{array} 
\right\}
\]
if $f$ is \textit{non-trivial},
and as $0$ if $f$ is \textit{trivial}.
When there is at least one $a$ satisfying all conditions,
then $\alpha$ is a well-defined non-negative real number.
If there are no such $a$, we say that $\alpha$ is not well-defined.
Let $U_r = U_r^{\alpha} =\{ |z| < r|w|^{\alpha}, |w| < r \}$.
The benefit of $\alpha$ is presented in the following lemma.

\begin{lemma} \label{main lem: d > 1}
Let $d \geq 2$.
If $\alpha$ is well-defined, then
$f(z,w) \sim (z^{\delta}, z^{\gamma} w^d)$ on $U_r^{\alpha}$ as $r \to 0$,
and $f(U_r^{\alpha}) \subset U_r^{\alpha}$ for small $r$. 
\end{lemma}

The notation $f \sim f_0$ means that
the ratios of the first and second components of $f$ and $f_0$ 
converge to $1$ on $U_r^{\alpha}$ as $r \to 0$.
Hence 
Lemma \ref{main lem: d > 1} says that the asymptotic behavior 
of $f$ on $U_r^{\alpha}$ when $r \to 0$ coincides with $f_0$,
where $f_0(z,w) = (z^{\delta}, z^{\gamma} w^d)$.
With the next theorem, we get a stronger result,
the existence of a conjugacy between $f$ and $f_0$.

\begin{theorem}\label{main thm: d > 1}
Let $d \geq 2$.
If $\alpha$ is well-defined, then
there is a biholomorphic map $\phi$ defined on $U_r^{\alpha}$,
with $\phi \sim id$ on $U_r^{\alpha}$ as $r \to 0$,
that conjugates $f$ to $(z,w) \to (z^{\delta}, z^{\gamma} w^d)$. 
\end{theorem}

We call $\phi$ the B\"{o}ttcher coordinate for $f$ in this paper. 
As in the one-dimensional case, 
it is obtained as the limit of the compositions of $f_0^{-n}$ and $f^n$.

For $d = 1$ 
we need the additional condition $\alpha < (\delta - 1)/ \gamma$
to get again B\"{o}ttcher coordinates.

\begin{lemma}\label{main lem: d = 1}
Let $d = 1$.
If $\alpha$ is well-defined and $\alpha < (\delta - 1)/ \gamma$, then
$f(z,w) \sim (z^{\delta}, b z^{\gamma} w)$ on $U_r^{\alpha}$ as $r \to 0$,
and $f(U_r^{\alpha}) \subset U_r^{\alpha}$ for small $r$. 
\end{lemma}

\begin{theorem}\label{main thm: d = 1}
Let $d = 1$.
If $\alpha$ is well-defined and $\alpha < (\delta - 1)/ \gamma$, then
there is a biholomorphic map $\phi$ defined on $U_r^{\alpha}$,
with $\phi \sim id$ on $U_r^{\alpha}$ as $r \to 0$,
that conjugates $f$ to $(z,w) \to (z^{\delta}, b z^{\gamma} w)$. 
\end{theorem}

Our results also hold for the nilpotent case.
We say that the germ $f : (\mathbb{C}^2, 0) \to (\mathbb{C}^2, 0)$ is nilpotent 
if $f(0)=0$ and the eigenvalues of $Df(0)$ are both zero.
If $f$ is nilpotent, then $f^2$ is superattracting.
Hence Lemmas \ref{main lem: d > 1} and \ref{main lem: d = 1} hold for $f^2$; 
these lemmas hold even for $f$ on  $U_r^{\alpha} \cap \{ |z|\leq r_1, |w| \leq r_2 \}$,
where $r_1$ is enough smaller than $r_2$. 
Consequently, Theorems \ref{main thm: d > 1} and \ref{main thm: d = 1} hold for $f$ itself.

Moreover, we can perturb $f$ slightly 
so that it is not skew product but our results hold.
Let $\tilde{p} (z,w) = z^{\delta} + \sum_{} a_l z^{n_l} w^{m_l}$,
where $n_l \geq \delta$, and $m_l \geq 1$ if $n_l = \delta$,
and let $q$ be the same as above.
Then, for the holomorphic germ of the form $f = (\tilde{p}, q)$,
we have the same lemma and theorem as in the skew product case.

The organization of the paper is as follows.
In Section 2 
we study the properties of the weight $\alpha$, 
and prove Lemmas \ref{main lem: d > 1} and \ref{main lem: d = 1}.
Assuming $d \geq 2$,
we prove that $\phi_n = f_0^{-n} \circ f^n$ is well-defined  
and converges uniformly to $\phi$ on $U_r^{\alpha}$ in Section 3,
and that $\phi$ is injective in Section 4. 
The optimality of $\alpha$ is shown by an example
at the end of Section 4. 
The case $d = 1$ is studied in Section 5.
Finally,
we slightly generalize our results to holomorphic germs in Section 6.

%\newpage
%%%%%%%%%%%%%%%%%%%%%%%%%%%%%%%%%%%%%%%%%%%%%%%%%%%%%%%%%%
%%%%%%%%%%%%%%%%%%%%%%%%%%%%%%%%%%%%%%%%%%%%%%%%%%%%%%%%%%
\section{Weights}

We now describe how to associate to a germ $f$ as above 
an interval $\mathcal{I}_f \subset \mathbb{R}$, so that, 
when $\alpha$ is well-defined,
it is given by $\alpha = \max \{ \inf \mathcal{I}_f, 0 \}$.
The interval $\mathcal{I}_f$ provides a wider class of weights
for which all our results hold,
although it does not appear directly in the final conclusions in the introduction.

We define the interval $\mathcal{I}_f$ associated with $f$ as
\[
\mathcal{I}_f = 
\left\{ a \in \mathbb{R} \ \Big| 
\begin{array}{lcr}
a (a \gamma + d) \leq a \delta \text{ and } 
a \gamma + d \leq a n_j + m_j \\ 
\text{for any $j$}
\text{ such that } b_j \neq 0 
\end{array} 
\right\}.
\]
Let $U_{r_1, r_2}^a = \{ |z| < r_1 |w|^{a}, |w| < r_2 \} \cap \{ |z| < r_2 \}$.
We remark that,
unlike the definition of $U_r^{\alpha}$ in the introduction,
this set needs to be intersected with $\{ |z| < r_2 \}$
because $a$ can be negative.

\begin{lemma} \label{weights: d > 1} 
Let $d \geq 2$.
For any number $a$ in $\mathcal{I}_f$, 
it follows that
$q(z,w) \sim z^{\gamma} w^d$ on $U_{r_1, r_2}^{a}$ as $r_1$, $r_2 \to 0$,
and $f(U_{r_1, r_2}^{a}) \subset U_{r_1, r_2}^{a}$ for small $r_1$ and $r_2$. 
\end{lemma}

\begin{proof} 
We first define $\eta (z,w) = (q(z,w) - z^{\gamma} w^d)/ z^{\gamma} w^d$
and show that $\eta \to 0$ on $U_{r_1,r_2}^{a}$ as $r_1$, $r_2 \to 0$,
which implies that $q(z,w) \sim z^{\gamma} w^d$ on $U_{r_1, r_2}^{a}$ as $r_1$, $r_2 \to 0$. 
Let $|z| = |cw^a|$ for any $a$ in $\mathcal{I}_f$.
Then $U_{r_1, r_2}^a \subset \{ |c| < r_1, |w| < r_2 \}$ and 
\[
|\eta|
= \left| \sum \frac{b_j z^{n_j} w^{m_j}}{z^{\gamma} w^{d}} \right| 
= \left| \sum \frac{b_j (cw^{a})^{n_j} w^{m_j}}{(cw^{a})^{\gamma} w^{d}} \right| 
= \left| \sum \frac{b_j c^{n_j} w^{a n_j +m_j}}{c^{\gamma} w^{a \gamma + d}} \right|
\]
\[
\leq \sum |b_j| |c|^{n_j - \gamma} |w|^{(a n_j + m_j) - (a \gamma + d)}.
\]
The conditions $n_j \geq \gamma$ and $a n_j + m_j \geq a \gamma + d$ ensure that
the left-hand side is a power series in $|c|$ and $|w|$, 
and so converges on $\{ |c| < r_1, |w| < r_2 \}$.
Moreover, 
at least one of the inequalities $n_j > \gamma$ 
or $a n_j + m_j > a \gamma + d$ holds
since $n_j \geq \gamma$, and $m_j > d$ if $n_j = \gamma$.
In other words,
$n_j - \gamma \geq 1$ or $(a n_j + m_j) - (a \gamma + d) \geq 1$ holds.
Therefore,
$\eta \to 0$ on $U_{r_1,r_2}^{a}$ as $r_1$, $r_2 \to 0$.

For the invariance of $U_{r_1, r_2}^{a}$,
it is enough to show that
$|p(z)| < r_1 |q(z,w)|^{a}$ for any $(z,w)$ in $U_{r_1, r_2}^a$.
Since
\[
\left| \frac{p(z)}{q(z,w)^{a}} \right| 
\sim \left| \frac{z^{\delta}}{(z^{\gamma} w^{d})^{a}} \right| 
= \left| \frac{(cw^{a})^{\delta}}{((cw^{a})^{\gamma} w^{d})^{a}} \right|
= |c|^{\delta - a \gamma} |w|^{a \delta - a (a \gamma + d)}
\]
on $U_{r_1, r_2}^a$,
we need the conditions $\delta - a \gamma \geq 0$ 
and $a \delta \geq a (a \gamma + d)$.
However,
the condition $\delta - a \gamma \geq 0$ follows from 
the condition $a \delta \geq a (a \gamma + d)$
because $d \geq 2$.
In fact, it follows that $\delta - a \gamma \geq 2$;
if $a \leq 0$ then $\delta - a \gamma \geq \delta \geq 2$, and
if $a > 0$ then $\delta - a \gamma \geq d \geq 2$.
Hence $|p(z)/q(z,w)^{a}| \leq C \cdot |c|^2 \leq |c| < r_1$
for some constant $C$ and sufficiently small $r_1$.
\end{proof}

\begin{lemma}\label{weights: d = 1} 
Let $d = 1$.
For any number $a$ in $\mathcal{I}_f$, 
if $a < (\delta - 1)/ \gamma$,
then $q(z,w) \sim b z^{\gamma} w$ on $U_{r_1, r_2}^{a}$ as $r_1$, $r_2 \to 0$,
and $f(U_{r_1, r_2}^{a}) \subset U_{r_1, r_2}^{a}$ for small $r_1$ and $r_2$. 
\end{lemma}

\begin{proof} 
The proof of the asymptotic behavior of $q$ is 
similar to the proof of Lemma \ref{weights: d > 1}.
To prove the invariance of $U_{r_1, r_2}^{a}$,
we need to check that $\delta - a \gamma \geq 0$.
In fact,
the additional condition $a < (\delta - 1)/ \gamma$ implies that $\delta - a \gamma >1$. 
Hence $|p/q^{a}| \leq C \cdot |c|^{1 + \varepsilon} \leq |c| < r_1$
for some constant $C$ and small $r_1$, 
where $\varepsilon = \delta - a \gamma - 1 > 0$. 
\end{proof} 

We show that 
Lemmas \ref{weights: d > 1} and \ref{weights: d = 1} induce
Lemmas \ref{main lem: d > 1} and \ref{main lem: d = 1}, respectively,
at the end of this section.

Let us describe $\mathcal{I}_f$ more practically. 
Let $\alpha_0 = (\delta - d)/\gamma$, 
which is derived from the first condition in the definition of $\mathcal{I}_f$. 
The second condition $a \gamma + d \leq a n_j + m_j$ implies that
\[
a \geq \frac{d - m_j}{n_j - \gamma} 
\]
if $n_j > \gamma$.
We define $m_f$ as
\[
\sup \left\{ \frac{d - m_j}{n_j - \gamma} \ \Big|  
\begin{array}{lr}
b_j\neq 0 \text{  and } n_j > \gamma
\end{array} 
\right\},
\]
where this value is set as $- \infty$ 
when the supremum is taken over the empty set.
Note that $\mathcal{I}_f \subset [m_f, \infty)$.
If $f$ is \textit{trivial}, then $m_f \leq 0$.
If $f$ is \textit{non-trivial}, then $m_f > 0$ and %, moreover,
we can replace the supremum to the maximum in the definition of $m_f$. 

If $f$ is \textit{trivial}, then we can describe $\mathcal{I}_f$ as follows,
where  $m_f \leq 0$.
%%%%%%%%%%%%%%%%%%%%%%%%%%%
\begin{center}
\begin{tabular}{|c|c|c|} \hline
\rule[-5pt]{0pt}{18pt} $f$ \textit{trivial} & $\gamma = 0$ & $\gamma \neq 0$ \\ \hline
\rule[-5pt]{0pt}{18pt} $\delta > d$ & $[0, \infty)$ & $[0, \alpha_0]$ \\ \hline
\rule[-5pt]{0pt}{18pt} $\delta = d$ & $[m_f, \infty)$ & $\{ 0 \}$ \\ \hline
\rule[-5pt]{0pt}{18pt} $\delta < d$ & $[m_f, 0]$ & $[\max \{ m_f, \alpha_0 \}, 0]$ \\ \hline
\end{tabular}
\end{center}
%%%%%%%%%%%%%%%%%%%%%%%%%%%
In particular, $\mathcal{I}_f$ is always non-empty if $f$ is \textit{trivial}.
If $f$ is \textit{non-trivial}, then we can describe $\mathcal{I}_f$ as follows,
where $m_f > 0$.
%%%%%%%%%%%%%%%%%%%%%%%%%%%
\begin{center}
\begin{tabular}{|c|c|c|} \hline
\rule[-5pt]{0pt}{18pt} $f$ \textit{non-trivial} & $\gamma = 0$ & $\gamma \neq 0$ \\ \hline
\rule[-5pt]{0pt}{18pt} $\delta > d$ & $[m_f, \infty)$ & $[m_f, \alpha_0]$ or $\emptyset$ \\ \hline
\rule[-5pt]{0pt}{18pt} $\delta = d$ & $[m_f, \infty)$ & $\emptyset$ \\ \hline
\rule[-5pt]{0pt}{18pt} $\delta < d$ & $\emptyset$ & $\emptyset$ \\ \hline
\end{tabular}
\end{center}
%%%%%%%%%%%%%%%%%%%%%%%%%%%
Note that $\mathcal{I}_f$ can be empty if $f$ is \textit{non-trivial}.
For the case $\delta > d$ and $\gamma \neq 0$,
the interval $\mathcal{I}_f$ is equal to $[m_f, \alpha_0]$ if $m_f \leq \alpha_0$ 
and is empty if $m_f > \alpha_0$.

We may restrict our attention to non-negative weights for our theorems,
although negative weights make sense as in Lemmas \ref{weights: d > 1} and \ref{weights: d = 1}.
Then the assumption $a \geq 0$ reduces 
the condition $a (a \gamma + d) \leq a \delta$
to the condition $a \gamma + d \leq \delta$ unless $a = 0$,
which induces the definition of $\alpha$.
The interval of non-negative numbers 
that satisfy the conditions in the definition of $\alpha$,
coincides with $\mathcal{I}_f \cap [0, \infty)$ if $\delta \geq d$.
For any case, it follows 
that $\alpha$ is well-defined if and only if $\mathcal{I}_f$ is not empty, and
that 
\[
\alpha = \min \mathcal{I}_f \cap [0, \infty) = \max \{ \inf \mathcal{I}_f, 0 \}
\] 
if it is well-defined. 
If $f$ is \textit{trivial}, then $\alpha = 0$.
The next table summarizes the relations 
between $\alpha$ and $m_f$ in the \textit{non-trivial} case.
%%%%%%%%%%%%%%%%%%%%%%%%%%%
\begin{center}
\begin{tabular}{|c|c|c|} \hline
\rule[-5pt]{0pt}{18pt} $f$ \textit{non-trivial} & $\gamma = 0$ & $\gamma \neq 0$ \\ \hline
\rule[-5pt]{0pt}{18pt} $\delta > d$ & $m_f$ & $m_f$ or $\nexists$ \\ \hline
\rule[-5pt]{0pt}{18pt} $\delta = d$ & $m_f$ & $\nexists$ \\ \hline
\rule[-5pt]{0pt}{18pt} $\delta < d$ & $\nexists$ & $\nexists$ \\ \hline
\end{tabular}
\end{center}
%%%%%%%%%%%%%%%%%%%%%%%%%%%
The notation $m_f$ in the table means that $\alpha$ is well-defined
and coincides with $m_f$.
The notation $\nexists$ means that $\alpha$ is not well-defined.

We are now ready to show 
Lemmas \ref{main lem: d > 1} and \ref{main lem: d = 1}. 

\begin{proof}[Proof of Lemmas \ref{main lem: d > 1} and \ref{main lem: d = 1}] 
We may assume that $\mathcal{I}_f \neq \emptyset$  
since $\alpha$ is well-defined.
If $f$ is \textit{trivial}, then $\alpha = 0 \in \mathcal{I}_f$.
If $f$ is \textit{non-trivial}, then $\alpha = m_f = \min \mathcal{I}_f > 0$.
Therefore,
Lemmas \ref{weights: d > 1} and \ref{weights: d = 1} imply
Lemmas \ref{main lem: d > 1} and \ref{main lem: d = 1}, respectively,
by taking $r$ as $\min \{ r_1, r_2 \}$.
\end{proof}

%\newpage
%%%%%%%%%%%%%%%%%%%%%%%%%%%%%%%%%%%%%%%%%%%%%%%%%%%%%%%%%%
%%%%%%%%%%%%%%%%%%%%%%%%%%%%%%%%%%%%%%%%%%%%%%%%%%%%%%%%%%
\section{Existence of the limit $\phi$ for the case $d \geq 2$}

In this section
we show that $\phi_n$ is well-defined and 
converges uniformly to $\phi$ on $U_r$ for the case $d \geq 2$,
where $\phi_n = f_0^{-n} \circ f^n$.
The proof is similar to \cite{u}.
In particular,
the estimate of $\| \Phi_{n+1} - \Phi_n \|$ is almost the same,
where $\Phi_n$ is a lift of $\phi_n$.
However,
we give a much more detailed description of $\phi_n$
and an explicit estimate of $\| \Phi - id \|$ with proofs in this paper,
where $\Phi$ is the limit of $\Phi_n$.
The biholomorphicity of $\phi$ will be proved in the next section,
which completes the proof of Theorem \ref{main thm: d > 1}.
The proof is different from the proof in \cite{u}. 

Before going into the proofs,
we remark on similarities and differences between this paper and \cite{u}.
Although the idea of assigning suitable weights
and the style of the main theorems are the same,
the choices of the major term of $q$,
the definitions of weights and invariant open sets are different. 
There are also several differences 
between settings and results in this paper and \cite{u}.  
The main theorems in this paper do not follow immediately from those in \cite{u}
because 
the two situations can not be connected with a simple conjugacy
even if we can extend the results for polynomial skew products in \cite{u} 
to holomorphic skew products defined near infinity.
%unlike the one-dimensional case.  

Let us prove that $\phi_n$ is well-defined, %Let us recall the outline of the proof,
assuming that $d \geq 2$ and that $\alpha$ is well-defined.
Let $p(z) = z^{\delta} (1 + \zeta (z))$
and $q(z,w) = z^{\gamma} w^d (1 + \eta (z,w))$;
Lemma \ref{main lem: d > 1} implies that 
$\zeta$ and $\eta$ are holomorphic on $U_r$ and
converge to $0$ as $r \to 0$. 
Then the first and second components of $f^n$ are written as
\[
z^{\delta^n} \prod_{j = 1}^{n} (1 + \zeta (p^{j - 1} (z)))^{\delta^{n-j}}
\text{ and}
\]
\[
z^{\gamma_n} w^{d^n} 
\prod_{j = 1}^{n - 1} (1 + \zeta (p^{j - 1} (z)))^{\gamma_{n-j}} 
\prod_{j = 1}^{n} (1 + \eta (f^{j - 1} (z,w)))^{d^{n-j}}, 
\]
where $\gamma_n = \sum_{j=1}^{n} \delta^{n-j} d^{j-1} \gamma$.
Using $\zeta$ and $\eta$,
we can also describe $\phi_n$ explicitly.

\begin{proposition}
We can define $\phi_n$ as follows:
\[ 
\phi_n (z,w)
= \left( z \cdot \prod_{j = 1}^{n} \sqrt[\delta^j]{1 + \zeta (p^{j - 1} (z))},
w \cdot \prod_{j = 1}^{n} \frac{\sqrt[d^j]{1 + \eta (f^{j - 1} (z,w))}}
{\sqrt[(\delta d)^j]{\{ 1 + \zeta (p^{j - 1} (z)) \}^{\gamma_j}}} \right),
\]
which is well-defined and so holomorphic on $U_r$.
\end{proposition}

\begin{proof}
Formally, 
$f_0^{-n} (z,w) = (z^{1 / \delta^n}, z^{- \gamma_n / \delta^n d^n} w^{1 / d^n})$
and we can define the first and second components of $\phi_n$ as
\[ 
\left\{ z^{\delta^n} \prod_{j = 1}^{n} (1 + \zeta (p^{j - 1} (z)))^{\delta^{n-j}} \right\}^{1 / \delta^n}
= z \cdot \prod_{j = 1}^{n} \sqrt[\delta^j]{1 + \zeta (p^{j - 1} (z))} \ \ 
\text{ and}
\]
\[ 
\left\{ \dfrac{ z^{\gamma_n} w^{d^n} 
\prod_{j = 1}^{n - 1} (1 + \zeta (p^{j - 1} (z)))^{\gamma_{n-j}} \prod_{j = 1}^{n} (1 + \eta (f^{j - 1} (z,w)))^{d^{n-j}} }
{ \left\{ z^{\delta^n} \prod_{j = 1}^{n} (1 + \zeta (p^{j - 1} (z)))^{\delta^{n-j}} \right\}^{\gamma_n / \delta^n} } \right\}^{1/d^n}
\]
\[
= w \cdot \left\{ \dfrac{ \prod_{j = 1}^{n} (1 + \eta (f^{j - 1} (z,w)))^{d^{n-j}} }
{ \left\{ \prod_{j = 1}^{n - 1} (1 + \zeta (p^{j - 1} (z)))^{\gamma_n/ \delta^j - \gamma_{n-j}} \right\} 
\left( 1 + \zeta (p^{n - 1} (z)))^{\gamma_n/ \delta^n} \right) } \right\}^{1/d^n}.
\]
Lemma \ref{lem for phi_n} below gives the explicit formula of $\phi_n$ above,
and Lemma \ref{main lem: d > 1} ensures that
$\phi_n$ is well-defined and so holomorphic on $U_r$.
\end{proof}

\begin{lemma}\label{lem for phi_n}
For any $1 \leq j \leq n-1$, it follows that
\[
\dfrac{\gamma_n}{\delta^j d^n} - \dfrac{\gamma_{n-j}}{d^n}
=  \dfrac{\gamma_j}{(\delta d)^j}.
\]
\end{lemma}

\begin{proof}
If $\delta \neq d$, 
then $\gamma_n = (\delta^n - d^n) \gamma /(\delta - d)$ and so
\[
\dfrac{\gamma_n}{\delta^j d^n} - \dfrac{\gamma_{n-j}}{d^n}
= \dfrac{\delta^n - d^n}{\delta^j d^n} \cdot \dfrac{\gamma}{\delta - d}
- \dfrac{\delta^{n-j} - d^{n-j}}{d^n} \cdot \dfrac{\gamma}{\delta - d}
\]
\[
= \left( \dfrac{1}{d^j} - \dfrac{1}{\delta^j} \right) \cdot \dfrac{\gamma}{\delta - d}
= \dfrac{\delta^j - d^j}{(\delta d)^j} \cdot \dfrac{\gamma}{\delta - d}
=  \dfrac{\gamma_j}{(\delta d)^j}.
\]
If $\delta = d$, then then $\gamma_n = n d^{n-1} \gamma$ and so
\[
\dfrac{\gamma_n}{\delta^j d^n} - \dfrac{\gamma_{n-j}}{d^n}
= \dfrac{n \gamma}{d^{j+1}} - \dfrac{(n-j) \gamma}{d^{j+1}}
= \dfrac{j \gamma}{d^{j+1}}
= \dfrac{j d^{j-1} \gamma}{d^{2j}}
= \dfrac{\gamma_j}{(\delta d)^j}.
\]
\end{proof}

In order to prove the uniform convergence of $\phi_n$,
we lift $f$ and $f_0$ to $F$ and $F_0$ 
by the exponential product $\pi (z,w) = (e^z, e^w)$; 
that is, 
$\pi \circ F = f \circ \pi$ and $\pi \circ F_0 = f_0 \circ \pi$.
More precisely,
we define
\[
F(Z, W) = (\delta Z + \log (1 + \zeta (e^Z)),
\gamma Z + dW + \log (1 + \eta (e^Z, e^W)))
\]
and $F_0 (Z,W) = (\delta Z, \gamma Z + d W)$;
let  $F_0 = (P_0, Q_0)$.
By Lemma \ref{main lem: d > 1},
we may assume %that
\[
\| F - F_0 \| <\varepsilon \text{ on } \pi^{-1} (U_r)
\]
for any small $\varepsilon > 0$, taking $r$ small enough.
Similarly, 
we can lift $\phi_n$ to $\Phi_n$ so that 
the equation $\Phi_n = F_0^{-n} \circ F^n$ holds; thus $\Phi_0 = id$ and, for any $n \geq 1$,
\[
\Phi_n (Z, W) = \left( \frac{1}{\delta^n} P_n(Z), 
\frac{1}{d^n} Q_n(Z,W) - \frac{\gamma_n}{\delta^n d^n} P_n(Z) \right), 
\] 
where $(P_n(Z),Q_n(Z,W)) = F^n(Z,W)$.
Let $\Phi_n = (\Phi_n^1, \Phi_n^2)$. 
Then
\[
|\Phi_{n+1}^1 - \Phi_n^1|
= \left| \frac{P_{n+1}}{\delta^{n+1}} - \frac{P_n}{\delta^n} \right|
= \frac{|P_{n+1} - \delta P_n|}{\delta^{n+1}} < \frac{1}{\delta^{n+1}} \varepsilon
\]
since $|P_{n+1} - \delta P_n| = |P(P_n)- P_0 (P_n)| = |Z \circ (F-F_0)(F^n)| < \varepsilon$, and
\[
|\Phi_{n+1}^2 - \Phi_n^2|
= \left| \left\{ \frac{Q_{n+1}}{d^{n+1}} - 
\frac{\gamma_{n+1} P_{n+1}}{\delta^{n+1} d^{n+1}} \right\}
- \left\{ \frac{Q_n}{d^n} - 
\frac{\gamma_n P_n}{\delta^n d^n} \right\} \right|
\]
\[
= \left| \frac{Q_{n+1}}{d^{n+1}} 
- \frac{\gamma P_n}{d^{n+1}} - \frac{Q_n}{d^n} \right| 
+ \left| \frac{\gamma_{n+1} P_{n+1}}{\delta^{n+1} d^{n+1}}  
- \frac{\gamma_n P_n}{\delta^n d^n} - \frac{\gamma P_n}{d^{n+1}} \right|
\]
\[
= \frac{|Q_{n+1} - (\gamma P_n + d Q_n)|}{d^{n+1}} 
+ \frac{\gamma_{n+1} |P_{n+1} - \delta P_n|}{\delta^{n+1} d^{n+1}} 
< \frac{1}{d^{n+1}} \varepsilon + \frac{\gamma_{n+1}}{\delta^{n+1} d^{n+1}} \varepsilon
\]
since $|Q_{n+1} - (\gamma P_n + d Q_n)| = |Q(F^n) - Q_0(F^n)| = |W \circ (F-F_0)(F^n)| < \varepsilon$.
Hence $\Phi_n$ converges uniformly to $\Phi$.
In particular, 
we can estimate $\| \Phi - id \|$ as follows. 

\begin{lemma}
It follows that
\[
\| \Phi - id \| < \max \left\{ \frac{1}{\delta - 1}, 
\frac{1}{d-1} + \frac{\gamma}{\delta - d} 
\left( \frac{1}{d - 1} - \frac{1}{\delta - 1} \right) \right\} \varepsilon
\text{ if } \delta \neq d, \text{ and}
\]
\[
\| \Phi - id \| < \left\{ \frac{1}{d - 1} + \frac{\gamma}{(d-1)^2} \right\} \varepsilon
\text{ if } \delta = d.
\]
\end{lemma}

\begin{proof}
Since $\| \Phi - id \| = \max \{ | \Phi^1 - Z|, | \Phi^2 - W| \}$,
where $\Phi = (\Phi^1, \Phi^2)$,
\[
\| \Phi - id \| \leq \max \left\{ \sum_{n=0}^{\infty} | \Phi_{n+1}^1 - \Phi_n^1 |, \sum_{n=0}^{\infty} | \Phi_{n+1}^2 - \Phi_n^2 | \right\} 
\]
\[
<\max \left\{  \sum_{n=0}^{\infty} \frac{1}{\delta^{n+1}},  
\sum_{n=0}^{\infty} \frac{1}{d^{n+1}} +  \sum_{n=0}^{\infty} \frac{\gamma_{n+1}}{\delta^{n+1} d^{n+1}} \right\} \varepsilon.
\]
If $\delta \neq d$, then $\gamma_n = (\delta^n - d^n) \gamma /(\delta - d)$ and so
\[
\sum_{n=0}^{\infty} \frac{\gamma_{n+1}}{\delta^{n+1} d^{n+1}} 
= \sum_{n=1}^{\infty} \frac{\gamma_{n}}{\delta^{n} d^{n}}
= \sum_{n=1}^{\infty} \frac{\gamma}{\delta - d} \left( \frac{1}{d^{n}} - \frac{1}{\delta^{n}} \right)
= \frac{\gamma}{\delta - d} \left( \frac{1}{d - 1} - \frac{1}{\delta - 1} \right).
\]
If $\delta = d$, then $\gamma_n = n d^{n-1} \gamma$ and so
\[
\sum_{n=1}^{\infty} \frac{\gamma_{n}}{\delta^{n} d^{n}}
= \sum_{n=1}^{\infty} \frac{n d^{n-1} \gamma}{d^{2n}}
= \frac{\gamma}{d} \cdot \sum_{n=1}^{\infty} \frac{n}{d^{n}}
= \frac{\gamma}{d} \cdot \frac{d}{(d-1)^2}
=  \frac{\gamma}{(d-1)^2}.
\]
\end{proof}

By the inequality $|e^z/e^w - 1| \leq |z - w| e^{|z - w|}$,
the uniform convergence of $\Phi_n$ translates into that of $\phi_n$.
Therefore,
$\phi$ is holomorphic on $U_r - \{ z = 0 \}$,
which extends to $U_r$ by  Riemann's removable singularity theorem.
In particular, if $|\Phi - id| < \varepsilon$, 
then $|\phi - id| < \varepsilon e^{\varepsilon} |id|$.
Therefore, 
$\phi \sim id$ on $U_r^{}$ as $r \to 0$.

%%%%%%%%%%%%%%%%%%%%%%%%%%%%%%%%%%%%%%%%%%%%%%%%%%%%%%%%%%
%%%%%%%%%%%%%%%%%%%%%%%%%%%%%%%%%%%%%%%%%%%%%%%%%%%%%%%%%%
\section{Injectivity of $\phi$ and optimality of $\alpha$ for the case $d \geq 2$} %Biholomorphicity

We continue the proof of Theorem \ref{main thm: d > 1}.
In the previous section we showed that 
$\phi$ is well-defined and so holomorphic on $U_r$.
However,
unlike the one-dimensional case,
the injectivity of $\phi$ does not follow immediately
because the domain $U_r$ may not be a neighborhood of the origin.
In this section we prove that,
after shrinking $r$ if necessary,
the map $\phi$ is actually injective on $U_r$.
More precisely,
the property $\phi \sim id$ suggests the injectivity of $\phi$,
which is ensured by Rouch\'e's theorem. 

Let $\phi = (\phi_1, \phi_2)$ and $U_{r_1, r_2} = \{ |z| < r_1 |w|^{\alpha}, |w| < r_2 \}$.
For simplicity,
we may assume that $\alpha > 0$ and 
that the function $\phi_1$ in $z$ is injective 
because it is conformal at the origin. 
Let us fix small $\varepsilon$, $r_1$ and $r_2$ such that
$|\zeta|$, $|\eta| < \varepsilon$ on $U_{r_1, r_2}$ 
and $f(U_{r_1, r_2}) \subset U_{r_1, r_2}$.
Then $\| F - F_0 \| < \log (1 + \varepsilon)$ on $\pi^{-1} (U_{r_1, r_2})$,
where $F$ is the lift of $f$ by $\pi (Z,W) = (e^Z, e^W)$ and
\[
\pi^{-1} (U_{r_1, r_2}) = \{ \mathrm{Re} (Z - \alpha W) < \log r_1, \mathrm{Re} W < \log r_2 \}.
\]
Let $\Phi (Z,W) = (\Phi_1 (Z), \Phi_Z (W))$ be the lift of $\phi$,
which is holomorphic on $\pi^{-1} (U_{r_1, r_2})$.
The injectivity of $\phi_1$ derives that of $\Phi_1$ because $\Phi_1 \sim id$.
We prove the injectivity of  $\Phi_Z$ in Proposition \ref{biholo of lift} below;
then the injectivity of $\Phi$ derives that of $\phi$ because  $\Phi \sim id$.
Recall that $|\Phi_Z - id| < C \tilde{\varepsilon}$,
where $\tilde{\varepsilon} = \log (1 + \varepsilon)$ and  
\[
C = \frac{1}{d-1} + \frac{\gamma}{\delta - d} 
\left( \frac{1}{d - 1} - \frac{1}{\delta - 1} \right)
\text{ or } C = \frac{1}{d-1} + \frac{\gamma}{(d-1)^2} 
\]
if $\delta \neq d$ or $\delta = d$.
Let $V_Z = V \cap (\{ Z \} \times \mathbb{C})$ 
and $V'_Z = V' \cap (\{ Z \} \times \mathbb{C})$,
where
\[
V = \pi^{-1} (U_{r_1, r_2})
= \left\{ \frac{\mathrm{Re} Z}{\alpha} - \frac{\log r_1}{\alpha} < \mathrm{Re} W < \log r_2 \right\}
\text{ and}
\]
\[
V' =
\left\{ \frac{\mathrm{Re} Z}{\alpha} - \frac{\log r_1}{\alpha} + 2C \tilde{\varepsilon} 
< \mathrm{Re} W < \log r_2 - 2C \tilde{\varepsilon} \right\}
\subset V.
\]

\begin{proposition}\label{biholo of lift}
Let $\alpha > 0$. Then $\Phi_Z$ is injective on $V'_Z$ for any fixed $Z$. 
\end{proposition}

\begin{proof}
Let $W_1$ and $W_2$ be two points in $V'_Z$
such that $\Phi_Z (W_1) = \Phi_Z (W_2)$,
and show that $W_1 = W_2$. 
Define $g (W) = \Phi_Z (W) - \Phi_Z (W_1)$ and $h (W) = W - \Phi_Z (W_1)$.
Then $|g - h| = |\Phi_Z - id| < C \tilde{\varepsilon}$ on $V_Z$.
By the definition of $V_Z$ and $V'_Z$,
there is a smooth, simply closed curve $\Gamma$ in $V_Z$
whose distances from $W_1$ and $W_2$ are greater than $C \tilde{\varepsilon}$.
Hence $|h| \geq \text{dist} (\Phi_Z(\Gamma), \partial V_Z)
\geq 2C \tilde{\varepsilon} -C \tilde{\varepsilon} = C \tilde{\varepsilon}$ on $\Gamma$.
Therefore,
$|g - h| < |h|$ on $\Gamma$.
Rouch\'e's theorem implies that 
the number of zero points of $g$ is exactly one
in the region surrounded by $\Gamma$;
thus $W_1 = W_2$. 
\end{proof}

\begin{proposition}\label{biholo of phi}
Let $\alpha > 0$. Then $\phi$ is injective on 
\[
\left\{ \frac{|z|}{|w|^{\alpha}} < \frac{r_1}{(1 + \varepsilon)^{2 \alpha C}}, \
|w| < \frac{r_2}{(1 + \varepsilon)^{2C}} \right\}.
\]
\end{proposition}

\begin{proof}
Since $\Phi_1$ and $\Phi_Z$ are injective for any $Z$ 
by Proposition \ref{biholo of lift}, we deduce that $\Phi$ is injective on $V'$.
Hence
$\phi$ is injective on $\pi (V')$ because  $\Phi \sim id$,
where $\pi (V') = \{ |z/w^{\alpha}| < r'_1, |w| < r'_2 \}$
for some constants $r'_1$ and $r'_2$.
Indeed,
$r'_1 = r_1/(1 + \varepsilon)^{2 \alpha C}$
and $r'_2 = r_2/(1 + \varepsilon)^{2C}$
since $(\log r'_1)/\alpha = (\log r_1)/\alpha - 2C \tilde{\varepsilon}$
and $\log r'_2 = \log r_2 - 2C \tilde{\varepsilon}$.
\end{proof}

\begin{remark}
By similar arguments,
it follows that $F$ is injective on 
\[
\left\{ \frac{Re Z}{\alpha} - \frac{\log r_1}{\alpha} +  \frac{2 \tilde{\varepsilon}}{d} 
< Re W < \log r_2 - \frac{2 \tilde{\varepsilon}}{d} \right\}.
\]
Hence $F^n$, $\Phi_n$ and $\Phi$ are injective on the same region. 
This region is bigger than $V'$ 
since $C \geq 1/(d-1) > 1/d$.
Therefore,
we have a bigger region that ensures the injectivity of $\phi$. 
\end{remark}

%%%%%%%%%%%%%%%%%%%%%%%%%%%%%%%%%%%%%%%%%%%%%%%%%%%%%%%%%%
%\section{An example and optimality of $\alpha$}

We next provide an example which indicates the optimality of $\alpha$.
It is a family of polynomial skew products that
are semiconjugate to polynomial products,
which contains the example $f(z,w) = (z^2, w^2 + z^4)$ in the introduction.
See also \cite{u-fiberwise} and \cite[Section 10]{u} for such maps.

\begin{example}\label{example: optimality}
Let $f(z,w) = (z^d, w^d + c z^{ld})$ and $l = A/B$,
where $d \geq 2$, $A \geq 1$ and $B$ is a divisor of $d$.
Then $f$ is semiconjugate to a product
$g(z,w) = (z^d, w^d + c)$ by $\pi (z,w) = (z^B, z^A w): \pi \circ g = f \circ \pi$.
We can construct the B\"{o}ttcher coordinate 
that conjugates $f$ to $f_0 (z,w) = (z^d, w^d)$ as follows.
Let $\varphi_g$ be the B\"{o}ttcher coordinate
%defined on $\{ |w| > R \}$ for large $R$
%that conjugates $w \to w^d + c$ to $w \to w^d$.
for $w \to w^d + c$ near infinity;
it is defined on $\{ |w| > R \}$ for large $R$
and conjugates $w \to w^d + c$ to $w \to w^d$.
Then $\phi_g (z,w) = (z, \varphi_g (w))$ is
a biholomorphic map %defined on $\{ |w| > R \}$
that conjugates $g$ to $g_0 (z,w) = (z^d, w^d)$.
Consequently,
$\phi_f = \pi \circ \phi_g \circ \pi^{-1}$ or, equivalently, 
$\phi_f (z,w) = \left( z, z^{l} \varphi_g \left( w/z^{l} \right) \right)$ is the required map;
it is a well-defined biholomorphic map defined on $\{ |w| > R|z|^{l} \}$
that conjugates $f$ to $f_0$.
\end{example}

Let us explain the optimality of $\alpha$, using this example.
Note that $\alpha = 1/l = B/A$ and $\mathcal{I}_f = \{ a \geq \alpha \}$.
As stated above,
the B\"{o}ttcher coordinate $\phi_f$ exists on $U_r^{\alpha}$ for small $r$.
This also follows from Theorem \ref{main thm: d > 1}.
Whereas 
we can replace $\alpha$ in Theorem \ref{main thm: d > 1}
with any $a \geq \alpha$ that belongs to $\mathcal{I}_f$,
we can not replace it with any $a < \alpha$;
that is,
$\phi_f$ can not extend from $U_r^{\alpha}$ to $U_r^a$ for any $a < \alpha$.
In fact, 
if $\phi_f$ extended to $U_r^a$ for some $a < \alpha$,
then $\varphi_g$ could extend to $\mathbb{C}$,
because the closure of $\pi^{-1} (U_r^a)$ includes the $w$-axis. 
However,
$\varphi_g$ can not extend to a region larger than the attracting basin of infinity,
except the special case $c = 0$.
Moreover,
it seems that the invariance of $U_r^a$ does not hold for any $a < \alpha$.

%%%%%%%%%%%%%%%%%%%%%%%%%%%%%%%%%%%%%%%%%%%%%%%%%%%%%%%%%% 
%%%%%%%%%%%%%%%%%%%%%%%%%%%%%%%%%%%%%%%%%%%%%%%%%%%%%%%%%%
\section{The case $d = 1$}

We extend our ideas and results for the case $d \geq 2$ to the case $d = 1$;
we prove Theorem \ref{main thm: d = 1}. 
The proof of the uniform convergence of $\phi_n$ is different from the previous case
because the sum of $d^{-n}$ does not converge anymore.
The condition $\alpha < (\delta - 1)/ \gamma$ is necessary for Theorem \ref{main thm: d = 1},
which is shown by an example at the end of this section.

We remark that $f$ is rigid of Class 4 in \cite{f} if $d = 1$
and hence it follows that
$f$ is conjugate to $(z,w) \to (z^{\delta}, b z^{\gamma} w + R(z))$
for some polynomial $R$.
Furthermore, we can remove $R$ thanks to Theorem \ref{main thm: d = 1}. 
We also remark that the case $d = 0$ exists,
but it is not treated in this paper because the map $f_0 (z,w) = (z^2, bz)$ is not dominant.

Let us prove Theorem \ref{main thm: d = 1}.
Since the investigation of the second components of maps 
is the essential part for proofs,
we sometimes omit the expressions of the first components hereafter. 
In a similar fashion to the case $d \geq 2$,
let $\eta = (q - bz^{\gamma}w)/bz^{\gamma}w$.
Then $|Q - Q_0| = |\log (1 + \eta)|$.
Since we may assume that $|\eta| < 1$,
\[
|Q - Q_0| \leq \log (1 + |\eta|) \leq |\eta|
\text{ and so } |Q(F^n) - Q_0(F^n)|  \leq |\eta (F^n)|.
\]
To prove the uniform convergence of $\phi_n$,
we show that $|\eta (F^n)|$ or, equivalently, $|\eta (f^n)|$ 
decreases rapidly as $n \to \infty$ in Lemma \ref{lem2: d = 1}.
First, we claim that $f^n$ contracts $U_r$ rapidly.
Since the origin is superattracting,
it is clear that $f^n$ contracts a small bidisk rapidly;
e.g., 
$f^n (\{ |z|<r, |w|<r \}) \subset \{ |z|<r/2^n, |w|<r/2^n \}$.
Moreover,
the same contraction holds for $U_r$,
where $U_r =\{ |z| < r|w|^{\alpha}, |w| < r \}$.

\begin{lemma}\label{lem1: d = 1}
Let $d = 1$.
If $\alpha$ is well-defined and $\alpha < (\delta - 1)/ \gamma$, 
then $f^n(U_r) \subset U_{r/2^n}$ for small $r$.
\end{lemma}

\begin{proof}
By Lemma \ref{main lem: d = 1},
for any small $\varepsilon$
there is $r$ such that
\[
|p(z)| < (1 + \varepsilon) |z^{\delta}| \text{ and } 
(1 - \varepsilon) |bz^{\gamma}w| < |q(z,w)| < (1 + \varepsilon) |bz^{\gamma}w| 
\]
on $U_r$.
Let $|z| = |cw^{\alpha}|$.
Then $U_r \subset \{ |c| < r, |w| < r \}$ and
\[
\left| \dfrac{p(z)}{q(z,w)^{\alpha}} \right|
< \dfrac{1 + \varepsilon}{(1 - \varepsilon)^{\alpha}} \cdot \left| \dfrac{z^{\delta}}{(bz^{\gamma}w)^{\alpha}} \right|
= \dfrac{1 + \varepsilon}{(1 - \varepsilon)^{\alpha}} \cdot \left| \dfrac{(cw^{\alpha})^{\alpha}}{\{ b(cw^{\alpha})^{\gamma}w \}^{\alpha}} \right|
\]
\[
= \dfrac{1 + \varepsilon}{(1 - \varepsilon)^{\alpha}} \cdot \dfrac{1}{|b|^{\alpha}} 
\cdot |c|^{\delta - \alpha \gamma} |w|^{\alpha \delta - \alpha (\alpha \gamma + 1)}. 
\]
By assumption,
$\delta - \alpha \gamma > 1$ and $\alpha \{ \delta - (\alpha \gamma + 1) \} \geq 0$.
Therefore,
shrinking $r$ so that
$(1 + \varepsilon)r^{\delta - \alpha \gamma - 1}/(1 - \varepsilon)^{\alpha} |b|^{\alpha} < 1/2$,
we obtain that
\[
\left| p/q^{\alpha} \right| < |c|/2 < r/2.
\]
In addition, since
$|q(z,w)| < (1 + \varepsilon) |b(cw^{\alpha})^{\gamma}w| =  (1 + \varepsilon) |b| |c|^{\gamma} |w|^{\alpha \gamma} \cdot |w|$,
\[
| q | < |w|/2 < r/2
\]
for $r$ such that $(1 + \varepsilon) |b| r^{\gamma (\alpha + 1)} < 1/2$.
This implies that 
\[
f( \{ |c| < r, |w| < r \} ) \subset \{ |c| < r/2, |w| < r/2 \};
\text{ that is, } f(U_r) \subset U_{r/2}.
\]
By repeating this calculation,
it follows that 
\[
f^n( \{ |c| < r, |w| < r \} ) \subset \{ |c| < r/2^n, |w| < r/2^n \};
\text{ that is, } f^n(U_r) \subset U_{r/2^n}.
\]
\end{proof}

Lemma \ref{lem1: d = 1} derives the uniform estimate of $|\eta (f^n)|$ on  $U_r$.

\begin{lemma}\label{lem2: d = 1}
Let $d = 1$.
If $\alpha$ is well-defined and $\alpha < (\delta - 1)/ \gamma$, 
then 
\[
|\zeta (p^n)| \leq \dfrac{C_1 r}{2^n} \text{ and } |\eta (f^n)| \leq \dfrac{C_2 r}{2^n}
\]
on $U_r$
for some constants $C_1$ and $C_2$.
\end{lemma}

\begin{proof}
Let $|z| = |cw^{\alpha}|$.
Then 
\[
|\eta|
= \left| \sum \frac{b_j z^{n_j} w^{m_j}}{bz^{\gamma} w} \right| 
\leq \sum \left| \dfrac{b_j}{b}  \right| \cdot |c|^{n_j - \gamma} |w|^{(a n_j + m_j) - (a \gamma + 1)}.
\]
By assumption,
$n_j \geq \gamma$ and $\alpha n_j + m_j \geq \alpha \gamma + 1$.
Moreover, 
at least one of the inequalities $n_j - \gamma \geq 1$ 
or $(\alpha n_j + m_j) - (a \gamma + 1) \geq 1$ holds.
Hence there exist constants $A$ and $B$ such that
$|\eta| \leq A |c| + B |w|$.
It then follows from Lemma \ref{lem1: d = 1} that 
$|\eta (f^n)| \leq A r/2^n + B r/2^n = (A + B)r/2^n$ on $U_r$.
\end{proof}

Now we are ready to prove the uniform convergence of $\phi_n$.

\begin{proposition}
Let $d = 1$.
If $\alpha$ is well-defined and $\alpha < (\delta - 1)/ \gamma$, 
then $\phi_n$ converges uniformly to $\phi$ on $U_r$,
and $\phi \sim id$ on $U_r^{}$ as $r \to 0$.
\end{proposition}

\begin{proof}
It is enough to show the uniform convergence of $\Phi_n^2$.
By Lemma \ref{lem2: d = 1},
\[
|\Phi_{n+1}^2 - \Phi_n^2|
\leq \frac{|Q(F^n) - Q_0(F^n)|}{d^{n+1}} 
+ \frac{\gamma_{n+1} |P(P^n) - P_0(P^n)|}{\delta^{n+1} d^{n+1}} 
\]
\[
\leq |\eta (F^n)| + \frac{\gamma}{\delta - 1} |\zeta (P^n)|
< \left( C_2 + \frac{\gamma}{\delta - 1} C_1 \right)  \frac{r}{2^n}.
\]
\end{proof}

The proof of the injectivity of $\phi$ is the same as the case $d \geq 2$.

\begin{proposition}
Let $d = 1$.
If $\alpha$ is well-defined and $\alpha < (\delta - 1)/ \gamma$, 
then $\phi$ is injective on $U_r$ for small $r$.
\end{proposition}

Finally,
we  exhibit the following two examples.
The first example satisfies all the conditions of Theorem \ref{main thm: d = 1},
and the second one does not.

\begin{example}
Let $f(z,w) = (z^2, b zw + z^3)$.
Then $\alpha = 1/2 < (\delta - 1)/ \gamma = 1$.
By Theorem \ref{main thm: d = 1},
there exists the B\"{o}ttcher coordinate on $U_r$ for small $r$,
that conjugates $f$ to $f_0 (z,w) = (z^{\delta}, b zw)$. 
\end{example}

Note that if $\alpha$ is well-defined, then 
$\alpha \leq (\delta - 1)/ \gamma$. 
The second example satisfies the equation $\alpha = (\delta - 1)/ \gamma$.

\begin{example}
Let $f(z,w) = (z^2, b zw + z^2)$.
Then $\alpha = (\delta - 1)/ \gamma = 1$, and
$f$ is semiconjugate to $g(z,w) = (z^2, b w + 1)$ 
by $\pi (z,w) = (z, zw): \pi \circ g = f \circ \pi$.
Moreover, if $b \neq 1$, then
$f$ is conjugate to $f_0 (z,w) = (z^2, b zw)$ 
by $h_f$, and 
$g$ is conjugate to $g_0 (z,w) = (z^2, b w)$  
by $h_g$,
where $h_f(z,w) = (z, w + z/(1 - b))$
and $h_g(z,w) = (z, w + 1/(1 - b))$. 
\end{example}

For this example,
Theorem \ref{main thm: d = 1} does not hold at least if $b = 1$.
In fact,
if we had a B\"{o}ttcher coordinate that conjugated $f$ to $f_0 (z,w) = (z^2, zw)$,
then $g$ should be conjugate to $g_0 (z,w) = (z^2, w)$.
However,
the translation $w \to w + 1$ can not be conjugate to the identity $w \to w$.
Although an conjugacy $h_f$ exists if $b \neq 1$,
the dynamics is different from our case.
In particular,
the second component of $g_0$ in this example is affine,
whereas the second component of $g_0$ in Example \ref{example: optimality}
is $w^d$, where $d \geq 2$, 
and so it has a superattracting fixed point at infinity.
%%%%%
We can slightly generalize this example
to $f(z,w) = (z^{\delta}, b z^{\delta - 1} w + z^{\delta})$.
Since $\gamma = \delta - 1$, 
again $\alpha = (\delta - 1)/ \gamma = 1$, 
and $f$ is semiconjugate to $g(z,w) = (z^{\delta}, bw + 1)$ 
by $\pi (z,w) = (z^{\delta - 1}, z^{\gamma} w)$. %\pi \circ g = f \circ \pi$.

%%%%%%%%%%%%%%%%%%%%%%%%%%%%%%%%%%%%%%%%%%%%%%%%%%%%%%%%%% 
%%%%%%%%%%%%%%%%%%%%%%%%%%%%%%%%%%%%%%%%%%%%%%%%%%%%%%%%%%
\section{A generalization to holomorphic germs}

Until now
we have dealt with a germ 
of holomorphic skew product of the form $f(z,w) = (p(z),q(z,w))$
such that $p(z) = z^{\delta} + a_{\delta + 1} z^{\delta + 1} + \cdots$
and 
\[
q(z,w) = b z^{\gamma} w^d + \sum_{} b_j z^{n_j} w^{m_j},
\]
where $b \neq 0$, $\gamma \leq n_j$, and $d < m_j$ if $\gamma = n_j$.
Since the origin is a superattracting fixed point, 
$\delta \geq 2$, $\gamma + d \geq 2$ and $n_j + m_j \geq 2$.
In this section
we perturb $p$ to a holomorphic germ $\tilde{p}$ in $z$ and $w$
such that $\tilde{p}(z,w) = a(w) z^{\delta} + a_{\delta + 1}(w) z^{\delta + 1} + \cdots$,
where $a(0) = 1$.
In other words,
\[
\tilde{p} (z,w) = z^{\delta} + \sum_{} a_l z^{n_l} w^{m_l},
\]
where $n_l \geq \delta$, and $m_l \geq 1$ if $n_l = \delta$.
Let $f(z, w) = (\tilde{p}(z, w), q(z, w))$ hereafter.

We first construct a biholomorphic map $\phi$ that conjugate $f$ to $f_0$
by arguments similar to the skew product case,
where $f_0(z,w) = (z^{\delta}, bz^{\gamma} w^d)$.
It is more difficult to prove the injectivity of $\phi$
because $f$ does not preserve the family of fibers anymore.
We then give another proof of $f$ being conjugate to $f_0$.
In fact, it follows from \cite{r} that $f$ is conjugate to 
a holomorphic germ of the form $\tilde{f}(z,w) = (z^{\delta}, \tilde{q}(z,w))$
for some $\tilde{q}$.

In a similar fashion to the skew product case,
we define the rational number $\alpha$ associated with $f$ as
\[
\alpha =
\min 
\left\{ a \geq 0 \ \Big| 
\begin{array}{lcr}
a \gamma + d \leq \delta, \
a \delta \leq a n_l + m_l \text{ and } 
a \gamma + d \leq a n_j + m_j \\ 
\text{for any $j$ such that } b_j \neq 0
\end{array} 
\right\}
\]
if $f$ is \textit{non-trivial},
and as $0$ if $f$ is \textit{trivial}.
We remark that 
the condition $a \delta \leq a n_l + m_l$ is trivial and can be removed
since $n_l \geq \delta$ and $a \geq 0$,
although the interval $\mathcal{I}_{f}$
may differ whether we add the condition.
%Hence the coefficients of $\tilde{p}$ do not matter to $\alpha$.
Hence the weights of the skew product $(p,q)$ and 
the holomorphic germ $f = (\tilde{p},q)$ are the same.
Moreover,
the weights of $f$ and $\tilde{f}$ are also the same,
as stated in Lemma \ref{lemma: two weights} below. 

Let us construct a B\"{o}ttcher coordinate for $f$.
%as the skew product case.
Since $\tilde{p}(z,w) \sim z^{\delta}$ on  the neighborhood $\{ |z| < r, |w| < r \}$ as $r \to 0$,
we have the following lemma.

\begin{lemma}
Let $\alpha$ be well-defined.
If $d \geq 2$ or if $d = 1$ and $\alpha < (\delta - 1)/ \gamma$, then
$f(z,w) \sim (z^{\delta}, bz^{\gamma} w^d)$ on $U_r^{}$ as $r \to 0$,
and $f(U_r) \subset U_r$ for small $r$.
\end{lemma}

This lemma induces the existence of the limit of the compositions of 
$f_0^{-n}$ and $f^n$ as previous cases,
where $f_0(z,w) = (z^{\delta}, bz^{\gamma} w^d)$.

\begin{theorem}\label{thm for general case}
Let $\alpha$ be well-defined.
If $d \geq 2$ or if $d = 1$ and $\alpha < (\delta - 1)/ \gamma$, then
there is a biholomorphic map $\phi$ defined on $U_r^{}$,
with $\phi \sim id$ on $U_r^{}$ as $r \to 0$,
that conjugates $f$ to $(z,w) \to (z^{\delta}, bz^{\gamma} w^d)$. 
\end{theorem}

The proof of the existence of $\phi$ is similar to the skew product case.
The difficult part of the proof is the injectivity of $\phi$.
Since $\phi$ is clearly injective if  $\alpha = 0$,
we may assume that $\alpha > 0$ hereafter.
Let us state the idea of the proof of  the injectivity of $\phi$.
As in Section 4,
we prove that the lift $\Phi$ of $\phi$ is injective,
which implies the injectivity of $\phi$ because  $\Phi \sim id$.
For the skew product case,
we applied Rouch\'e's theorem to $\Phi$ restricted to a vertical line
in order to show that $\Phi_Z$ is injective,
where $\Phi = (\Phi_1, \Phi_Z)$.
Since we may assume that $\Phi_1$ is injective,
this implies that $\Phi$ is injective.
On the other hand,
in this section
we apply Rouch\'e's theorem to $\Phi$ restricted to a line,
which may not be vertical, as follows.
Let $\Phi$ be well-defined and holomorphic on $V$,
and take a sufficiently small region $V'$ in $V$.
Let $A_1$ and $A_2$ be two points in $V'$
such that $\Phi (A_1) = \Phi (A_2)$.
Applying Rouch\'e's theorem to $\Phi$ restricted to
the intersection of $V$ and the line $L$ passing through $A_1$ and $A_2$,
we can show that $A_1 = A_2$.

The point is taking a smaller region $V'$ in $V$ such that
$L \cap (V \setminus V')$ has a suitable width for any line $L$ intersecting $V'$,
as in Section 4.
Recall that 
\[
V = \left\{ \frac{\mathrm{Re} Z}{\alpha} - \frac{\log r_1}{\alpha} < \mathrm{Re} W < \log r_2 \right\},
\]
and let $\| \Phi - id \| < \varepsilon$. 
Then the following region is what we need:  
\[
V' = 
\left\{ \frac{\mathrm{Re} Z}{\alpha} - \frac{\log r_1}{\alpha} 
+ \frac{1 + \alpha}{\alpha} \cdot 2 {\varepsilon} 
< \mathrm{Re} W < \log r_2 - 2 {\varepsilon} \right\}.
\]
Let us illustrate where the constant $(1 + \alpha)/\alpha$ comes from.
First, consider everything in $\mathbb{R}^{2}$.
Let $L = \{ y = mx \}$, $\mathcal{V} = \{ y > x/\alpha \}$
and $\mathcal{V}' = \{ y > x/\alpha + R \cdot 2 {\varepsilon} \}$ for a constant $R$,
where $(x,y) \in \mathbb{R}^{2}$ and $m \in \mathbb{R}$.
If $|m| \geq 1$, 
then we take the projection $\pi_2$ to the second coordinate,
and require that 
the length of the interval $\pi_2 (L \cap (\mathcal{V} \setminus \mathcal{V}'))$ in $\mathbb{R}$
is greater than or equal to $2 {\varepsilon}$.
It is enough to consider the case $m = -1$,
since the length takes the minimum for this case.
By an elementary calculation in terms of two right-angled triangles,
it follows that, if $R = 1 + 1/\alpha$,
then the length coincides with $2 {\varepsilon}$.
If $|m| \leq 1$,
then we take the projection $\pi_1$ to the first coordinate.
By the same argument,
it follows that, if $R = 1 + 1/\alpha$,
then the length of  $\pi_1 (L \cap (\mathcal{V} \setminus \mathcal{V}'))$ 
is greater than or equal to $2 {\varepsilon}$. 
This sketch works for complex setting as well:

\begin{lemma}\label{dist lem}
Let $L$ be a line $\{ W = mZ + n \}$ which intersects $V'$.
Then
\[
\text{dist} (\pi_1^{-1} (L \cap V'), \partial \pi_1^{-1} (L \cap V)) \geq 2 {\varepsilon}
\text{ if } |m| \leq 1, \text{ and}
\]
\[
\text{dist} (\pi_2^{-1} (L \cap V'), \partial \pi_2^{-1} (L \cap V)) \geq 2 {\varepsilon}
\text{ if } |m| \geq 1,
\]
where $\pi_1$ and $\pi_2$ are the projections to $Z$ and $W$ coordinates, respectively.
\end{lemma}

\begin{proof}
Let $n = 0$ for simplicity.
We only prove the case $|m| \geq 1$.
Note that
\[
\pi_2^{-1} (L \cap V') = H \cap
\left\{ \mathrm{Re} W < \frac{1}{\alpha} \mathrm{Re} \frac{W}{m} 
- \frac{\log r_1}{\alpha} + \frac{1 + \alpha}{\alpha} \cdot 2 \varepsilon \right\}
\]
\[
= H \cap 
\left\{ \mathrm{Re} \{ (\alpha - 1/m) W \} < 
- \log r_1 + (1 + \alpha) 2 \varepsilon \right\},
\]
where $H = \{ \mathrm{Re} W < \log r_2 - 2 \varepsilon \}$.
It is enough to show that 
$\text{dist} (l_0, l_{\varepsilon}) \geq 2 \varepsilon$,
where $l_0 : \{ \mathrm{Re} \{ (\alpha - 1/m) W \} = 0 \}$
and $l_{ \varepsilon} : \{ \mathrm{Re} \{ (\alpha - 1/m) W \} =  (1 + \alpha) 2 \varepsilon \}$.
Actually, 
\[
\text{dist} (l_0, l_{\varepsilon}) 
= \frac{(1 + \alpha) 2 \varepsilon}{|\alpha - 1/m|} \geq 2 \varepsilon
\text{ since }
\left| \alpha - \frac{1}{m} \right| \leq \alpha + \frac{1}{|m|} \leq \alpha + 1.
\]
\end{proof}

Now we are ready to prove the injectivity of $\Phi$.

\begin{proposition}
The map $\Phi$ is injective on $V'$.
\end{proposition}

\begin{proof}
Let $\Phi (A_1) = \Phi (A_2)$ 
for points $A_1$ and $A_2$ in $V'$.
Let $L$ be the line passing through $w_1$ and $w_2$.
It is enough to consider the case $L = \{ W = mZ + n \}$.
Define $\tilde{\Phi}_1 = \pi_1 \circ \Phi \circ u$
and $\tilde{\Phi}_2 = \pi_2 \circ \Phi \circ v$,
where $u(Z) = (Z, mZ + n)$
and $v(W) = (W/m, W + n)$:
\begin{equation*}
\begin{CD}
\tilde{\Phi}_1 (\text{or} \ \tilde{\Phi}_2) : \text{preimage in } 
\mathbb{C}^{} @> u  (\text{or} \ v) >> L \cap V @> \Phi >> 
\mathbb{C}^{2} @> \pi_1  (\text{or} \ \pi_2) >> \mathbb{C}^{}.
\end{CD}
\end{equation*}
It then follows from Lemma \ref{dist lem} 
that $A_1 = A_2$,
by applying Rouch\'e's theorem to $\tilde{\Phi}_1$ or $\tilde{\Phi}_2$
if $|m| \leq 1$ or $|m| \geq 1$
as in Proposition \ref{biholo of lift}. 
\end{proof}

Finally,
we give another proof of Theorem \ref{thm for general case}.
The germ $f$ can be written as $(z^{\delta} (1 + \varepsilon (z,w)), q(z,w))$,
where $\varepsilon$ converges to $0$ as $z$ and $w$ tend to $0$.
Moreover, Theorem 1.3 in \cite{r} induces the following.

\begin{proposition}
The germ $f$ is conjugate to a holomorphic germ 
of the form $\tilde{f}(z,w) = (z^{\delta}, \tilde{q}(z,w))$ for some $\tilde{q}$.
\end{proposition}

\begin{proof}
We briefly review the proof in \cite{r} following a slightly different presentation.  
Define
\[
\phi_n (z,w) = \left( z \cdot \prod_{j=1}^n \sqrt[{\delta}^j]{1 + \varepsilon (f^{j-1}(z, w))}, w \right).
\]
Then $\phi_n$ is well-defined on a small neighborhood of the origin, and
\[
\phi_n \circ f = \tilde{f}_n \circ \phi_{n+1}
\] 
holds, where $\tilde{f}_n (z,w) = (z^{\delta}, q(\phi_{n+1}^{-1} (z,w)))$.
Since $\phi_n$ converges uniformly to $\phi_{\infty}$, 
it follows that 
$\phi_{\infty} \circ f = \tilde{f} \circ \phi_{\infty}$,
where $\tilde{f} (z,w) = (z^{\delta}, q(\phi_{\infty}^{-1} (z,w)))$.
\end{proof}

Since $\tilde{f}$ is skew product,
we can construct the B\"{o}ttcher coordinate $\tilde{\phi}$ 
defined on $U_r^{\tilde{\alpha}}$
that conjugates $\tilde{f}$ to $f_0$ as previous sections,
where $\tilde{\alpha}$ denotes the weight of $\tilde{f}$
and $f_0(z,w) = (z^{\delta}, bz^{\gamma} w^d)$.
Moreover,
the region $U_r^{\tilde{\alpha}}$ coincides with $U_r^{\alpha}$:

\begin{lemma}\label{lemma: two weights}
The weights $\alpha$ and $\tilde{\alpha}$ of $f$ and $\tilde{f}$ are the same.
\end{lemma}

\begin{proof}
We may write $\phi_{\infty} (z,w) = (z(1+u(z,w)),w)$
for a holomorphic germ $u$,
and so $\tilde{q} (z,w) = q(z(1+v(z,w)),w)$ 
for a holomorphic germ $v$
since  $\tilde{f} = \phi_{\infty} \circ f \circ \phi_{\infty}^{-1}$.
Let $b_j z^{n_j} w^{m_j}$ be a term in $q$.
Since $v$ is holomorphic in a neighborhood of the origin, 
the power series expansion of
the corresponding term $b_j \{ z(1+v(z,w)) \}^{n_j} w^{m_j}$ in $\tilde{q}$
can be expressed as
\[
b_j z^{n_j} w^{m_j} + \sum_{} b_{ij} z^{n_{ij}} w^{m_{ij}},
\]
where $n_{ij} \geq n_j$ and $m_{ij} \geq m_j$ for any $i$.
In particular,
$\tilde{q}$ has the same major term $b z^{\gamma} w^d$ as $q$.
If $f$ is \textit{trivial},
then $\tilde{f}$ is also \textit{trivial} and $\alpha = \tilde{\alpha} = 0$.
Let $f$ be \textit{non-trivial}.
Then $m_j < d$ and $n_j > \gamma$ for some $j$,
\[
\alpha =
\sup \left\{ \frac{d - m_j}{n_j - \gamma} \ \Big|  
\begin{array}{lr}
b_j\neq 0 \text{  and } n_j > \gamma
\end{array} 
\right\} > 0
\text{ and}
\]
\[
\tilde{\alpha} = 
\sup_{} \left\{ \frac{d - m_j}{n_j - \gamma}, \frac{d - m_{ij}}{n_{ij} - \gamma} \ \Big|  
\begin{array}{lr}
b_j\neq 0, n_j > \gamma \text{ and } b_{ij} \neq 0
\end{array} 
\right\} > 0.
\]
Since $d - m_{ij} \leq d - m_j$ and $n_j - \gamma \leq n_{ij} - \gamma$,
it follows that
\[
\frac{d - m_{ij}}{n_{ij} - \gamma} \leq \frac{d - m_j}{n_j - \gamma}
\]
if $n_j > \gamma$.
Therefore,
$\alpha = \tilde{\alpha}$.
\end{proof}

Consequently, the composition $\tilde{\phi} \circ \phi_{\infty}$
coincides with the B\"{o}ttcher coordinate $\phi$
in Theorem \ref{thm for general case},
that is defined on $U_r^{\alpha}$ and conjugates $f$ to $f_0$.

%    Bibliographies can be prepared with BibTeX using amsplain,
%    amsalpha, or (for "historical" overviews) natbib style.
\bibliographystyle{amsplain}
%    Insert the bibliography data here.

%%%%%%%%%%%%%%%%%%%%%%%%%%%%%%%%%%%%%%%%%%%%%%%%%%%%%%%%%%

\end{document}